\documentclass[a4paper, 12pt]{amsart}
\usepackage[english]{babel}
\usepackage[utf8]{inputenc}
\usepackage[T1]{fontenc}
\usepackage{ae}
\usepackage[pdftex]{hyperref}
\usepackage{amsmath}
\usepackage{amssymb}
\usepackage{amsthm}
\usepackage[shortlabels]{enumitem}

\newtheorem{thm}{Theorem}
\newtheorem{lem}[thm]{Lemma}
\newtheorem{cor}[thm]{Corollary}

\theoremstyle{definition}
\newtheorem{defn}[thm]{Definition}
\newtheorem{ex}[thm]{Example}

\newtheorem{rem}[thm]{Remark}

\theoremstyle{remark}

\newcommand{\gauss}[3]{\genfrac{[}{]}{0pt}{}{#1}{#2}_{#3}}
\DeclareMathOperator{\Der}{Der}
\DeclareMathOperator{\Res}{Res}
\DeclareMathOperator{\Red}{Red}
\DeclareMathOperator{\GF}{GF}
\DeclareMathOperator{\LS}{LS}
\DeclareMathOperator{\PGammaL}{P\Gamma L}

\title{Derived and Residual Subspace Designs}
\author{Michael Kiermaier}
\address{Michael Kiermaier\\Mathematisches Institut\\Universität Bayreuth\\D-95440 Bayreuth\\Germany} 
\email{michael.kiermaier@uni-bayreuth.de} 
\urladdr{http://www.mathe2.uni-bayreuth.de/michaelk/}

\author{Reinhard Laue}
\address{Reinhard Laue\\Institut für Informatik\\Universität Bayreuth\\D-95440 Bayreuth\\Germany} 
\email{laue@uni-bayreuth.de} 

\date{\today}

\begin{document}

\begin{abstract}
A generalization of forming derived and residual designs from $t$-designs to subspace designs is proposed.
A $q$-analog of a theorem by Van Trung, van Leijenhorst and Driessen is proven, stating that if for some (not necessarily realizable) parameter set the derived and residual parameter set are realizable, the same is true for the reduced parameter set.

As a result, we get the existence of several previously unknown subspace designs.
Some consequences are derived for the existence of large sets of subspace designs.
Furthermore, it is shown that there is no $q$-analog of the large Witt design.
\end{abstract}

\maketitle

\section{Introduction}
\subsection{History}
Let $V$ be a $v$-dimensional vector space over a finite field $\GF(q)$.
A $t$-$(v,k,\lambda)_q$ subspace design $\mathcal{D} = (V,{\mathcal B})$
consists of a set $\mathcal B$ of $k$-dimensional subspaces,
called blocks, such that each $t$-dimensional subspace of
$V$ lies in exactly $\lambda$ blocks.
This notion is a vector space analog of ordinary set-theoretic $t$-designs.
For that reason, subspace designs are also called $q$-analogs of designs.

The first reference about subspace designs appears to be \cite{Cameron74}, and the first actual subspace designs with $t\geq 2$ have been constructed in \cite{Thomas87}.
An introduction to subspace designs can be found in \cite[Day~4]{suzuki-fivedays}.

There has been a growing interest in subspace designs, recently.
New subspace designs for $t = 2$ and $t = 3$ have been constructed in \cite{BKL05, MBraun05, SBraun10, MBraun11}.
A major success was the discovery of a $2$-$(13,3,1)_2$ subspace design in \cite{BEOVW13}, which is the first $q$-analog of a Steiner system, and the $q$-analog to Teirlinck's theorem \cite{T}, stating that simple $t$-subspace designs exist for every value of $t$ \cite{FLV}.

Based on these results, it is natural to investigate further concepts in classical set-theoretic design theory for their applicability in the $q$-analog case.
In \cite{Kiermaier-Pavcevic}, intersection numbers for subspace designs are given.
In this article, we consider the fundamental constructions of derived and residual designs.

\subsection{Overview}
In the case of a set-theoretic $t$-$(v,k,\lambda)$ design ${\mathcal D} = (V,{\mathcal B})$ 
a point $x\in V$ is fixed and the blocks fall into classes of those
that contain or avoid $x$:
\begin{align*}
    \Der_x({\mathcal D}) & = (V\setminus \{x\},\{B\setminus \{x\}: x\in B \in {\mathcal B}\}) \\
    \Res_x({\mathcal D}) & = (V\setminus \{x\},\{B \in \mathcal{B} : x\not\in B\})
\end{align*}
Here, the derived design $\Der_x({\mathcal D})$ is a $(t-1)$-$(v-1,k-1,\lambda)$ design and the residual design $\Res_x({\mathcal D})$ is a $(t-1)$-$(v-1,k,\mu)$ design, where $\mu = \lambda\cdot \frac{v-t}{k-t+1}$.

While the above definition of the derived design is translated directly to the $q$-analog case, for the translation of the residual design we will start from the equivalent description
\[
	\Res_S(\mathcal{D}) = (S, \{B\in\mathcal{B} : B \subseteq S\})
\]
where $S$ is a $(v-1)$-subset of $V$, see Definition~\ref{def:der_res}.

It is worth mentioning that not all concepts of set-theoretic design theory do have a straightforward $q$-analog.
While the existence of a $q$-analog of the Fano plane is still an important unsolved problem, in Example~\ref{ex:witt} it will be shown that there is there is no $q$-analog of the large Witt design.

In set-theoretic design theory, there is a theorem found independently by Van Trung \cite{tvt}, van Leijenhorst \cite{vL}, and Driessen \cite{D}, stating that for any two set-theoretic designs with the parameters of the derived and the residual design of some (not necessarily realizable) parameter set, there is a design with the parameters of the reduced design.%
\footnote{We remark that in \cite{tvt,vL,D}, the involved designs are addressed by their numerical parameters.
However, there is no interpretation in terms of the derived, residual and reduced design parameters.}
In Theorem~\ref{tvtq}, a $q$-analog of this theorem will be given.
As an application, in Corollary~\ref{cor:new_designs} we get the hitherto unknown existence of subspace designs with the parameters
\[
    2\text{-}(8,4,\lambda)_2
    \qquad\text{where}\qquad
    \lambda\in\{63,84,147,168,189,252,273,294\}
\]
and
\[
    2\text{-}(10,4,\lambda)_2
    \qquad\text{where}\qquad
    \lambda\in\{1785, 1870, 3570, 3655, 5355\}\text{.}
\]
In Corollary~\ref{cor:ls}, the application of Theorem~\ref{tvtq} yields a $q$-analog of \cite[Lemma~4]{AN-K-Halvings}, which is a construction method for new large sets from known ones.

\section{Preliminaries}
\subsection{The subspace lattice}
In the following, we fix a prime power $q$ and a vector space $V$ over $\GF(q)$ of finite dimension $v$.
The lattice of all subspaces of $V$ will be denoted by $\mathcal{L}(V)$.
The set of all subspaces of $V$ of a fixed dimension $k$ is known as the \emph{Graßmannian} and will be denoted by $\gauss{V}{k}{q}$.
For simplicity, its elements will be called \emph{$k$-subspaces}.
The cardinality of $\gauss{V}{k}{q}$ is the Gaussian binomial coefficient
\[
\gauss{v}{k}{q}
= \prod_{i = 0}^{k-1} \frac{q^{v - i} - 1}{q^{i+1} - 1}
= \begin{cases}
     0 & \text{if } k > v\text{,} \\
	  \frac{(q^n - 1)(q^{n-1} - 1)\cdot\ldots\cdot(q^{n - k + 1} - 1)}{(q-1)(q^2 - 1)\cdot\ldots\cdot (q^k - 1)} & \text{otherwise.}
\end{cases}
\]
Because of
\[
	\lim_{q\to 1}\gauss{v}{k}{q} = \binom{v}{k}\text{,}
\]
the Gaussian binomial coefficients are considered the $q$-analogs of the binomial coefficients, see \cite{GoldmanRota70}.
Many identities for binomial coefficients have $q$-analogs for the Gaussian binomial coefficients.
We make use of the following identities
\[
	\gauss{n}{k}{q} = \gauss{n}{n-k}{q}
	\qquad\text{and}\qquad
	\gauss{n}{h}{q}\gauss{n-h}{k}{q} = \gauss{n}{k}{q} \gauss{n-k}{h}{q}
\]
and the $q$-Pascal triangle identities ($n \geq 1$)
\[
	\gauss{n}{k}{q} = \gauss{n-1}{k-1}{q} + q^k\gauss{n-1}{k}{q} = q^{n-k}\gauss{n-1}{k-1}{q} + \gauss{n-1}{k}{q}\text{.}
\]

\subsection{Subspace designs}
\begin{defn}
A pair $(V,\mathcal{B})$ with $\mathcal{B} \subseteq \gauss{V}{k}{q}$ is called a \emph{$t$-$(v,k,\lambda)_q$ (subspace) design}, if for each $T\in\gauss{V}{t}{q}$ there are exactly $\lambda$ elements of $\mathcal{B}$ containing $T$.
\end{defn}

According to a statement of Tits \cite{Tits}, combinatorics on sets can be seen as the limit case $q\to 1$ of combinatorics on vector spaces over $\GF(q)$, see also~\cite{Cohn}.
Indeed, many statements about subspace designs become true statements about set-theoretic designs when setting $q = 1$, and replacing notions about vector spaces by their set-theoretic counterpart.%
\footnote{Gaussian binomial coefficients are replaced by ordinary binomial coefficients, vector spaces and subspaces are replaced by sets and subsets, the dimension is replaced by the cardinality, etc.}

The following fact is the $q$-analog of a well-known property of block designs, which can be easily proven by a double counting argument:

\begin{lem}[{\cite[Lemma~4.1(1)]{Suzuki-1990}}]
\label{lem:lambda_s}
Let $D$ be a $t$-$(v,k,\lambda)_q$ design.
For each $s\in\{0,\ldots,t\}$, $D$ is a $s$-$(v,k,\lambda_s)_q$ design with
\[
\lambda_s = \lambda\cdot\frac{\gauss{v-s}{t-s}{q}}{\gauss{k-s}{t-s}{q}} = \lambda\cdot\frac{\gauss{v-s}{k-s}{q}}{\gauss{v-t}{k-t}{q}}\text{.}
\]
In particular, the number of blocks of $D$ is
\[
    \lambda_0
    = \lambda\cdot\frac{\gauss{v}{t}{q}}{\gauss{k}{t}{q}}
    = \lambda\cdot\frac{\gauss{v}{k}{q}}{\gauss{v-t}{k-t}{q}}\text{.}
\]
\end{lem}
In the situation $s = t-1$, the resulting subspace design is called the \emph{reduced} design.

For the existence of a $t$-$(v,k,\lambda)_q$ design, necessarily the \emph{integrality conditions} $\lambda_s \in\mathbb{Z}$ must be satisfied for all $s\in\{0,\ldots,t\}$.
If this is the case, we call the parameter set $t$-$(v,k,\lambda)_q$ \emph{admissible}, without requiring that the parameter set is \emph{realizable}, meaning that a $t$-$(v,k,\lambda)_q$ design actually exists.

As an example, the parameter set $2$-$(7,3,1)_q$ ($q$-analog of the Fano plane) and the parameter set $5$-$(12,6,1)_q$ ($q$-analog of the small Witt design) are admissible for any prime power $q$.
However, the question for the realizability of these parameter set is open for all values of $q$.

\begin{ex}
\label{ex:3-22_6_1}
We consider the parameters $3$-$(22,6,1)_q$.
Denoting the $n$-th cyclotomic polynomial by $\Phi_n\in\mathbb{Z}[X]$,
the integrality conditions yield that
\begin{multline*}
	\lambda_0
	= \frac{\gauss{22}{3}{q}}{\gauss{6}{3}{q}}
	= \frac{(q^{22} - 1)(q^{21} - 1)(q^{20} - 1)}{(q^6 - 1)(q^5 - 1)(q^4 - 1)} \\
	= \frac{\Phi_{22}(q)\Phi_{21}(q)\Phi_{20}(q)\Phi_{11}(q)\Phi_{10}(q)\Phi_{7}(q)}{\Phi_6(q)}
\end{multline*}
must be integral.
Using the fact that for any $x\in\mathbb{Z}$, $\gcd(\Phi_a(x),\Phi_b(x)) > 1$ can only happen if $a/b$ is the integral power of a prime, we see that $\lambda_0 \in\mathbb{Z}$ implies $q = 1$.
So in contrast to the set-theoretic case $q = 1$, for $q \geq 2$ the parameters $3$-$(22,6,1)_q$ are not admissible and there is no $3$-$(22,6,1)_q$ subspace design.
\end{ex}

Two subspace designs $(V,\mathcal{B})$ and $(V',\mathcal{B}')$ are called \emph{isomorphic} it there is a lattice isomorphism $\mathcal{L}(V) \to \mathcal{L}(V')$ mapping $\mathcal{B}$ to $\mathcal{B}'$.
By the fundamental theorem of projective geometry, the set of lattice isomorphisms $V \to V'$ is given by the set of bijective semilinear mappings $V \to V'$.
Furthermore, the group of lattice automorphisms of $\mathcal{L}(V)$ is isomorphic to the projective semilinear group $\PGammaL(V)$.

\subsection{Duality}
For an ordinary set-theoretic design $(V,\mathcal{B})$, the \emph{supplementary} design is defined as
\[
    (V,\{V\setminus B : B\in\mathcal{B}\})\text{.}
\]
It has the parameters
\[
    t\text{-}\left(v,v-k,\lambda\cdot \frac{\binom{v-k}{t}}{\binom{k}{t}}\right)\text{.}
\]

To get a $q$-analog of this construction, fix some non-singular bilinear form $\beta$ on the vector space $V$ over $\GF(q)$.
For $U\in\mathcal{L}(V)$ we denote the \emph{dual subspace} by
\[
    U^\perp = \{x\in V : \beta(x,y) = 0 \text{ for all }y\in U\}\text{.}
\]
For a $t$-$(v,k,\lambda)_q$ design $D = (V,\mathcal{B})$, the \emph{dual design} is defined as
\[
    D^\perp = (V,\{B^\perp : B\in \mathcal{B}\})\text{.}
\]
In~\cite[Lemma~4.2]{Suzuki-1990} it was shown that $D^\perp$ is a design with the parameters
\[
	t\text{-}\left(v,v-k,\lambda\cdot\frac{\gauss{v-k}{t}{q}}{\gauss{k}{t}{q}}\right)_q\text{.}
\]

In fact, the only required property is that $U \mapsto U^\perp$ is an antiautomorphism of the subspace lattice $\mathcal{L}(V)$ of $V$.
For two such antiautomorphisms $\phi$ and $\phi'$ the mapping $\phi^{-1}\circ\phi'$ is an automorphism of $\mathcal{L}(V)$.
So up to isomorphism, the dual design $D^\perp$ does not depend on the choice of the antiautomorphism.

\section{Derived and residual designs}
\begin{defn}
	\label{def:der_res}
	Let $D = (V,\mathcal{B})$ be a $t$-$(v,k,\lambda)_q$ design.
	For $U\in\gauss{V}{1}{q}$, the \emph{derived} design of $D$ in $U$ is defined as
	\[
	    \Der_U(D) = (V/U, \{B / U : B\in\mathcal{B}, U\leq B\})\text{.}
	\]
	For $H\in\gauss{V}{v-1}{q}$, the \emph{residual} design of $D$ in $H$ is defined as
	\[
	    \Res_H(D) = (H, \{B : B\in\mathcal{B}, B\leq H\})\text{.}
	\]
\end{defn}

For the special case of Steiner systems, the derived design was used in \cite[Th.~2]{Schwartz-Etzion-2002}.
To the authors' knowledge, the above notion of the residual design is new.

We point out that the derived subspace design is a factor design while the residual design is a subdesign of $D$.
Of course, there are many choices for $U$ and $H$, which may lead to non-isomorphic derived and residual designs.
However, the design parameters are the same for all derived and all residual parameters, respectively:

\begin{lem}
	\label{lem:der_res}
	With the notation as in Definition~\ref{def:der_res}, $\Der_U(D)$ is a $(t-1)$-$(v-1,k-1,\lambda)_q$ design and $\Res_H(D)$ is a $(t-1)$-$(v-1,k,\mu)_q$ design where
	\[
		\mu
		= \lambda\cdot\frac{\gauss{v-k}{1}{q}}{\gauss{k-t+1}{1}{q}}
		= \lambda\cdot \frac{q^{v-k} - 1}{q^{k-t+1} - 1}
		= \frac{\lambda_{t-1} - \lambda}{q^{k-t+1}}\text{.}
	\]
\end{lem}

\begin{proof}
	A $t$-subspace $T$ that contains $U$ lies in exactly $\lambda$
	blocks $B \in {\mathcal B}$ that also contain $U$. Factoring out $U$ yields the
	blocks of the claimed derived design on $V/U$.

	For $\Res_H(D)$, we have to count the blocks $B\in\mathcal{B}$ with $T\leq B\leq H$ for any $(t-1)$-subspace $T$ of $H$.
	According to \cite[Lemma~4.2]{suzuki-fivedays}, this number is $\lambda\cdot\left(\gauss{v-t}{k-t+1}{q}/\gauss{v-t}{k-t}{q}\right)$, which evaluates to the expressions given for $\mu$.
\end{proof}

\begin{rem}
\begin{enumerate}[(a)]
\item For $q = 1$, we get back the parameters of the derived and the residual design in the set-theoretic case.
\item In the above proof, the numbers $\mu_i^j$ of \cite[Lemma~4.2]{suzuki-fivedays} were used in the special case $i = t-1$, $j = 1$.
More general that lemma says that for nonnegative integers $i$ and $j$ with $i + j \leq t$ and fixed subspaces $I\in\gauss{V}{i}{q}$, $J\in\gauss{V}{j}{q}$, the number
\[
    \mu_i^j = \#\{B\in\mathcal{B} \mid I \leq B \leq J\}
\]
does not depend on the choice of $I$ and $J$
With the notion of reduced and residual designs, we can give the following alternative characterization:
The number $\mu_i^j$ is the parameter $\lambda$ of any design which arises from a $t$-$(v,k,\lambda)_q$ design by taking $j$ times the residual design and $t-i-j$ times the reduced design, no matter in which order the reducing and residual steps are performed and which subspaces $H$ are chosen for the residual steps.%
\footnote{The $i$-fold reduced and $j$-fold residual is a $(t-i-j)$-$(v-j, k, \mu_{t-i-j}^j)_q$ design.}
\end{enumerate}
\end{rem}

\begin{ex}
\label{ex:witt}
The famous large Witt design with parameters $5$-$(24,8,1)$ does not have a $q$-analog for any prime power $q\geq 2$.
Otherwise, taking the derived design twice would give a $3$-$(22,6,1)_q$ design in contradiction to Example~\ref{ex:3-22_6_1}.%
\footnote{Similar to Example~\ref{ex:3-22_6_1}, it can also be shown that the parameters $5$-$(24,8,1)_q$ are not admissible for any prime power $q\geq 2$.}
\end{ex}

\begin{defn}
	Let $t$-$(v,k,\lambda)_q$ be a (not necessarily admissible) parameter set.
	We define its
	\begin{enumerate}[(a)]
		\item \emph{reduced parameter set} $(t-1)$-$(v,k,\lambda_{t-1})$,
		\item \emph{derived parameter set} $(t-1)$-$(v-1,k-1,\lambda)$,
		\item \emph{residual parameter set} $(t-1)$-$(v-1,k,\lambda\cdot \frac{q^{v-k}-1}{q^{k-t+1}-1})$.
	\end{enumerate}
\end{defn}

\begin{lem}
\label{lem:lambda_delta_rho}
Let $t$-$(v,k,\lambda)_q$ be a parameter set and $s\in\{0,\ldots,t-1\}$.
Denoting the $\lambda_s$-parameter (see Lemma~\ref{lem:lambda_s}) of the derived parameter set by $\delta_s$ and that of the residual parameter set by $\rho_s$, we have
\[
	\lambda_s
	= \delta_s + q^{k-s} \rho_s 
	= q^{v-k} \delta_s + \rho_s\text{.}
\]
\end{lem}

\begin{proof}
It is straightforward to check that
\[
	\lambda_s = \lambda\cdot\frac{\gauss{v-s}{k-s}{q}}{\gauss{v-t}{k-t}{q}}\text{,}
	\qquad
	\delta_s = \lambda\cdot\frac{\gauss{v-s-1}{k-s-1}{q}}{\gauss{v-t}{k-t}{q}}
	\qquad\text{and}\qquad
	\rho_s = \lambda\cdot\frac{\gauss{v-s-1}{k-s}{q}}{\gauss{v-t}{k-t}{q}}\text{.}
\]
Now the claim follows from the $q$-Pascal triangle identities.
\end{proof}

\begin{lem}
	Let $t$-$(v,k,\lambda)_q$ be an admissible parameter set.
	Then its reduced, derived and residual parameter sets are admissible, too.
\end{lem}

\begin{proof}
	It is clear that the reduced parameter set is admissible.
	With the notation as in Lemma~\ref{lem:lambda_delta_rho}, $\delta_s = \lambda_{s+1}$ is an integer for all $s\in\{0,\ldots,t-1\}$, so the derived parameter set is admissible.
	Now by $\rho_s = \lambda_s - q^{v-k} \delta_s$, also the residual parameter set is admissible.
\end{proof}

With respect to this notion of duality, the derived and the residual design are dual concepts:
\begin{lem}
Let $D$ be a design on $V$, $U\in\gauss{V}{1}{q}$ and $H\in\gauss{V}{v-1}{q}$.
Then
\[
    \Der_U(D)^\perp \cong \Res_{U^\perp}(D^\perp)\qquad\text{and}\qquad \Res_H(D)^\perp \cong \Der_{H^\perp}(D^\perp)\text{.}
\]
\end{lem}

\begin{ex}
\label{ex:Fano_dual}
The dual of a $q$-analog of the Fano plane (a $2$-$(7,3,1)_q$ design) would be a $2$-$(7,4,q^2 + 1)_q$ design.
These are the derived and the residual parameter sets of the parameter set $3$-$(8,4,1)_q$.
The same is true in the set-theoretic case $q = 1$, where we know that all the designs actually exist.
\end{ex}

\section{A $q$-analog of a theorem by Van~Trung, van~Leijenhorst and Driessen}
\label{sect:tld}

By the discussion in the previous section, the admissibility of a parameter set implies the admissibility of its derived, residual and reduced parameter sets.
Realizability is propagated in the same way.
In this section, we study the consequences of the derived and the residual design parameters both being admissible (resp. realizable), without requiring the original parameter set to be admissible (resp. realizable).

\begin{lem}
	\label{tvtq_admissibility}
	Let $t$-$(v,k,\lambda)_q$ be a parameter set whose derived and residual parameter sets are admissible.
	Then $t$-$(v,k,\lambda)_q$ is admissible, too.
\end{lem}

\begin{proof}
	We use the notation as in the proof of Lemma~\ref{lem:lambda_delta_rho}.
	According to that lemma, the values $\lambda_s$ are integers for all $s\in\{1,\ldots,t-1\}$.
	Furthermore, $\lambda_t = \delta_{t-1}$ is an integer.
\end{proof}

\begin{thm}
\label{tvtq}
Let $t$-$(v,k,\lambda)_q$ be a parameter set whose derived and residual parameter sets are realizable.
Then its reduced parameter set is realizable, too.
\end{thm}

\begin{proof}
Let $V$ be a $v$-dimensional vector space over $\GF(q)$, $\bar{V}$ a $(v-1)$-dimensional vector space over $\GF(q)$ and $\varphi : V \to \bar{V}$ a surjective $\GF(q)$-linear map.
Then $U = \ker(\varphi)$ is a $1$-subspace of $V$.
By the preconditions, there exists a $(t-1)$-$(v-1,k-1,\lambda)_q$ design $D_{\Der} = (\bar{V},\mathcal{B}_{\Der})$ and a $(t-1)$-$(v-1,k,\mu)_q$ design $D_{\Res} = (\bar{V},\mathcal{B}_{\Der})$ on $\bar{V}$, where $\mu = \lambda\cdot\frac{q^{v - k} - 1}{q^{k - t + 1} - 1}$.
Let
\[
    \mathcal{B}_1 = \{\varphi^{-1}(\bar{B}) : \bar{B} \in \mathcal{B}_{\Der}\}
\]
and
\[
    \mathcal{B}_2 = \{K : K\text{ is complement of }U\text{ in }\varphi^{-1}(\bar{B}), \bar{B}\in \mathcal{B}_{\Res}\}\text{.}
\]
Both sets $\mathcal{B}_1$ and $\mathcal{B}_2$ consist of $k$-subspaces of $V$.
We remark that each block $B = \varphi^{-1}(\bar{B})\in \mathcal{B}_1$ is uniquely determined by $\bar{B}\in \mathcal{B}_{\Der}$, and each block $B\in \mathcal{B}_2$ uniquely determines its $\bar{B}\in \mathcal{B}_{\Res}$ as $\bar{B} = \varphi(B)$.
Furthermore, the sets $\mathcal{B}_1$ and $\mathcal{B}_2$ are clearly disjoint, since the elements of $\mathcal{B}_1$ contain $U$, while the elements of $\mathcal{B}_2$ do not.

Now we claim that $(V,\mathcal{B}_{\Red})$ with $\mathcal{B}_{\Red} = \mathcal{B}_1 \cup \mathcal{B}_2$ is a design with the reduced parameter set $(t-1)$-$(v,k,\lambda_{\Red})_q$ where by Lemma~\ref{lem:lambda_s}
\[
    \lambda_{\Red} = \lambda\cdot\frac{\gauss{v - (t-1)}{t - (t-1)}{q}}{\gauss{k - (t-1)}{t - (t-1)}{q}} = \lambda\cdot\frac{\gauss{v - t + 1}{1}{q}}{\gauss{k-t+1}{1}{q}} = \frac{q^{v-t+1}-1}{q^{k-t+1}-1}\text{.}
\]
For the verification, consider a $(t-1)$-subspace $T$ of $V$.
We count the blocks in $\mathcal{B}_{\Red}$ containing $T$.

If $U \leq T$, then all such blocks come from $\mathcal{B}_1$, and $T \leq B$ is equivalent to $\varphi(T) \leq \varphi(B)$.
Since the dimension of $\varphi(T)$ in $\bar{V}$ is $t-2$, the design $\mathcal{B}_{\Der}$ contains
\[
    \lambda\cdot\frac{\gauss{(v-1) - (t-2)}{(t-1) - (t-2)}{q}}{\gauss{(k-1)-(t-2)}{(t-1)-(t-2)}{q}}
    = \lambda_{\Red}
\]
blocks $\varphi(B) \geq \varphi(T)$, and each such block uniquely determines the preimage $B$.

Now assume $U \nleq T$.
For $B\in \mathcal{B}_1$, $T \leq B$ if and only if $\varphi(T) \leq \varphi(B)$ and $\varphi(B)$ is a block of $\mathcal{B}_{\Der}$.
Since $\varphi(T)$ has dimension $t-1$ in $\bar{V}$, there are $\lambda$ such blocks.
Furthermore, $B\in \mathcal{B}_2$ passes through $T$ if and only if $\bar{B} = \varphi(B)$ is a block of $\mathcal{B}_{\Res}$ containing the $(t-1)$-dimensional subspace $\varphi(T)$ and $B$ is a complement of $U$ in the $(k+1)$-dimensional space $\varphi^{-1}(\bar{B})$ such that $T\leq B$.
There are $\mu$ such blocks $\bar{B}$, and considering the situation modulo $T$, we see that for each $\varphi^{-1}(\bar{B})$ there are exactly $q^{1\cdot((k+1)-(t-1) - 1)} = q^{k-t+1}$ suitable complements $B$.%
\footnote{In general, a $u$-subspace $U$ of a $v$-dimensional vector space $V$ over $\GF(q)$ has exactly $q^{u(v-u)}$ complements in $V$.}

So in total, there are
\[
\lambda + q^{k-t+1}\mu
= \lambda\cdot\left(1 + \frac{q^{v-t+1}-q^{k-t+1}}{q^{k-t+1} - 1}\right)
= \lambda\cdot \frac{q^{v-t+1} - 1}{q^{k-t+1} - 1}
= \lambda_{\Red}
\]
blocks of $\mathcal{B}_{\Red}$ passing through $T$.
\end{proof}

\begin{ex}
By Theorem~\ref{tvtq} and Example~\ref{ex:Fano_dual}, the existence of a $q$-analog of the Fano plane (a $2$-$(7,3,1)_q$ design) would imply the existence of a design with the reduced parameters of $3$-$(8,4,1)_q$, which is a $2$-$(8,4,q^4 + q^2 + 1)_q$ design.
\end{ex}

\begin{rem}
\begin{enumerate}[(a)]
\item
If in Theorem~\ref{tvtq} the parameter set $t$-$(v,k,\lambda)_q$ is realizable as a design $\mathcal{D}$, the application of the construction to a derived and a residual design of $\mathcal{D}$ won't necessarily reproduce $\mathcal{D}$.

A counterexample is given by any $2$-$(13,3,1)_2$ Steiner system $\mathcal{D}$ which exists by \cite{BEOVW13}.
We assume that $\mathcal{D}$ arises as $\mathcal{B}_1 \cup \mathcal{B}_2$ by the above construction.
Take two distinct blocks $B_1, B_2\in\mathcal{B}_2$ with $\phi(B_1) = \phi(B_2) = \bar{B}$.
Then $B_1$ and $B_2$ are complements of $U$ in $\phi^{-1}(\bar{B})$.
So $\dim(B_1 + B_2) \leq \dim(\phi^{-1}(\bar{B})) = 4$ and therefore $\dim(B_1) \cap \dim(B_2) \geq 2$, which contradicts the Steiner system property of $\mathcal{D}$.
\item In the situation of Theorem~\ref{tvtq}, the parameter set $t$-$(v,k,\lambda)_q$ is admissible by Lemma~\ref{tvtq_admissibility}.
For ordinary block designs, it is known that the parameters $t$-$(v,k,\lambda)_q$ are not necessarily realizable.
It is natural to conjecture the same to be true in the $q$-analog case.
However, so far not a single admissible parameter set has been shown to be non-realizable.
\end{enumerate}
\end{rem}

\begin{cor}
\label{cor:new_designs}
The parameter sets
\[
    2\text{-}(8,4,\lambda)_2
    \qquad\text{with}\qquad
    \lambda\in\{63,84,147,168,189,252,273,294\}
\]
and
\[
    2\text{-}(10,4,\lambda)_2
    \qquad\text{with}\qquad
    \lambda\in\{1785, 1870, 3570, 3655, 5355\}\text{.}
\]
are realizable.
\end{cor}

\begin{proof}
Table~\ref{Param} lists in the first column certain admissible (but not known to be realizable) parameter sets $t$-$(v,k,\lambda)_q$ whose derived (column~2) and residual parameter set (column~3) are known to be realizable.
Now by Theorem~\ref{tvtq}, the reduced parameter set (column~4) is realizable.
\end{proof}

\begin{table}[ht]
\caption{Parameter sets of subspace designs}
\label{Param}
\centering$\begin{array}{|c|cc|cc|c|}
\hline
t\text{-}(v,k,\lambda)_q & \text{derived}      &   & \text{residual}     &    & \text{reduced}         \\ 
\hline
3\text{-}(8,4,3)_2 & 2\text{-}(7,3,3)_2  & \text{\cite{MBraun11}} & 2\text{-}(7,4,15)_2 &  \text{\cite{MBraun11}}   & 2\text{-}(8,4,63)_2    \\ 
3\text{-}(8,4,4)_2 & 2\text{-}(7,3,4)_2  & \text{\cite{MBraun11}} & 2\text{-}(7,4,20)_2 &  \text{\cite{MBraun11}}   & 2\text{-}(8,4,84)_2    \\ 
3\text{-}(8,4,7)_2 & 2\text{-}(7,3,7)_2  & \text{\cite{MBraun11}} & 2\text{-}(7,4,35)_2 &  \text{\cite{MBraun11}}   & 2\text{-}(8,4,147)_2   \\ 
3\text{-}(8,4,8)_2 & 2\text{-}(7,3,8)_2  & \text{\cite{MBraun11}} & 2\text{-}(7,4,40)_2 &  \text{\cite{MBraun11}}   & 2\text{-}(8,4,168)_2   \\ 
3\text{-}(8,4,9)_2 & 2\text{-}(7,3,9)_2  & \text{\cite{MBraun11}} & 2\text{-}(7,4,45)_2 &  \text{\cite{MBraun11}}   & 2\text{-}(8,4,189)_2   \\ 
3\text{-}(8,4,12)_2 & 2\text{-}(7,3,12)_2 & \text{\cite{MBraun11}} & 2\text{-}(7,4,60)_2 &  \text{\cite{MBraun11}}   & 2\text{-}(8,4,252)_2   \\ 
3\text{-}(8,4,13)_2 & 2\text{-}(7,3,13)_2 & \text{\cite{MBraun11}} & 2\text{-}(7,4,65)_2 &  \text{\cite{MBraun11}}   & 2\text{-}(8,4,273)_2   \\ 
3\text{-}(8,4,14)_2 & 2\text{-}(7,3,14)_2 & \text{\cite{MBraun11}} & 2\text{-}(7,4,70)_2 &  \text{\cite{MBraun11}}   & 2\text{-}(8,4,294)_2   \\ 
3\text{-}(10,4,21)_2 & 2\text{-}(9,3,21)_2 & \text{\cite{SBraun10}} & 2\text{-}(9,4,441)_2 &  \text{\cite{SBraun10}}  & 2\text{-}(10,4,1785)_2 \\ 
3\text{-}(10,4,22)_2 & 2\text{-}(9,3,22)_2 & \text{\cite{MBraun11}} & 2\text{-}(9,4,462)_2 &  \text{\cite{MBraun11}}  & 2\text{-}(10,4,1870)_2 \\ 
3\text{-}(10,4,42)_2 & 2\text{-}(9,3,42)_2 & \text{\cite{SBraun10}} & 2\text{-}(9,4,882)_2 &  \text{\cite{SBraun10}}  & 2\text{-}(10,4,3570)_2 \\ 
3\text{-}(10,4,43)_2 & 2\text{-}(9,3,43)_2 & \text{\cite{MBraun11}} & 2\text{-}(9,4,903)_2 &  \text{\cite{MBraun11}}  & 2\text{-}(10,4,3655)_2 \\ 
3\text{-}(10,4,63)_2 & 2\text{-}(9,3,63)_2 & \text{\cite{SBraun10}} & 2\text{-}(9,4,1323)_2 &  \text{\cite{SBraun10}} & 2\text{-}(10,4,5355)_2 \\
\hline
\end{array}$
\end{table}

To our knowledge, all the parameter sets in Corollary~\ref{cor:new_designs} were not known to be realizable before.

\section{Application to large sets}
\begin{defn}
A \emph{large set} $\LS_q[N](t,k,v)$ is partition of $\gauss{V}{k}{q}$ into $N$ subspace designs with the parameters $t$-$(v,k,\lambda)_q$.
More precisely, it is a collection
\[
\{(V,\mathcal{B}_i) : i\in\{1,\ldots,N\}\}
\]
of $t$-$(v,k,\lambda)_q$ designs such that $\{\mathcal{B}_i : i \in\{1,\ldots,N\}\}$ is a partition of $\gauss{V}{k}{q}$.
\end{defn}

Note that the parameter $\lambda$ does not appear in the parameter set $\LS_q[N](t,k,v)$ of a large set.
This is because under the definition of a large set, the parameter $\lambda = \gauss{v-t}{k-t}{q} / N$ is already determined by the other parameters.

By the properties of the dual design it is clear that the duals of the designs in a large set again form a large set.
So the existence of an $\LS_q[N](t,k,v)$ is equivalent to the existence of an $\LS_q[N](t,v-k,v)$, see also \cite{BKOW}.
From the discussion of the derived and residual designs above we obtain the following result.

\begin{cor}
If there exists an $\LS_q[N](t,k,v)$ with $t\ge 1$, then there also exists an $\LS_q[N](t-1,k-1,v-1)$ and an $\LS_q[N](t-1,k,v-1)$.
\end{cor}

\begin{proof}
If the given large set is defined on the vector space $V$ then we fix a
$1$-dimensional subspace $U$ to form the derived designs on $V/U$. These
form the claimed $\LS_q[N](t-1,k-1,v-1)$. If we fix a subspace $H$ of
codimension $1$ in $V$ then the residual designs of the given large set
on $H$ form the claimed $\LS_q[N](t-1,k,v-1)$.
\end{proof}

The construction in Theorem \ref{tvtq} can be used for the construction of large sets.

\begin{cor}[{$q$-analog of \cite[Lemma~4]{AN-K-Halvings}}]
\label{cor:ls}
If there exists an $\LS_q[N](t,k-1,v-1)$ and an $\LS_q[N](t,k,v-1)$ then
there exists an $\LS_q[N](t,k,v)$.
\end{cor}

\begin{proof}
Let $V$ be a $v$-dimensional vector space over $\GF(q)$, $\bar{V}$ a $(v-1)$-dimensional vector space over $\GF(q)$ and $\varphi : V \to \bar{V}$ a surjective $\GF(q)$-linear map.
On $\bar{V}$, let $\{D_{\Der}^{(1)},\ldots,D_{\Der}^{(N)}\}$ be an $\LS_q[N](t,k-1,v-1)$ and $\{D_{\Res}^{(1)},\ldots,D_{\Res}^{(N)}\}$ an $\LS_q[N](t,k,v-1)$.
As in the proof of Theorem~\ref{tvtq}, for any $i\in\{1,\ldots,n\}$ we use the mapping $\varphi$ to combine the two subspace designs $D_{\Der}^{(i)}$ and $D_{\Res}^{(i)}$ into a $t$-$(v,k,\lambda_{\Red})_q$ subspace design $D_{\Red}^{(i)}$ on $V$.
Clearly, the block sets of the designs $D_{\Red}^{(i)}$ form a partition of $\gauss{V}{k}{q}$.
\end{proof}

For $t\ge 2$ no such combinable pairs of large sets have been found so far.
There are $\LS_2[3](2,3,8)$ and $\LS_2[3](2,5,8)$, see \cite{BKOW}.
If an $\LS_2[3](2,4,8)$ could be found then Corollary~\ref{cor:ls} would imply the existence of large sets with the parameters $\LS_2[3](2,4,9)$, $\LS_2[3](2,5,9)$ and $\LS_2[3](2,5,10)$.

\section*{Acknowledgement}
The authors thank  Thomas Feulner for stimulating discussions on this topic.


\begin{thebibliography}{10}
\bibitem{AN-K-Halvings}
{\sc S. Ajoodani-Namini and G. B. Khosrovashahi}, 
{\em More on halving the complete designs},
Discrete Mathematics {\bf 135} (1994), 29--37.

\bibitem{BKL05}
{\sc Michael Braun, Adalbert Kerber and Reinhard Laue},
{\em Systematic construction of $q$-analogs of designs},
Designs, Codes, Cryptography {\bf 34}, (2005), 55--70.

\bibitem{MBraun05}
{\sc Michael Braun},
{\em Some new designs over finite fields},
Bayreuther Math. Schr. {\bf 74} (2005), 58--68.

\bibitem{MBraun11}
{\sc Michael Braun},
{\em Designs over finite fields -- revisited},
Fq10, Ghent (2011).

\bibitem{BEOVW13}
{\sc Michael Braun, Tuvi Etzion, Patric R. \"Ostergard, Alexander Vardi and Alfred Wassermann},
{\em Existence of $q$-analogs of {S}teiner systems},
arXiv: 1304.1462.

\bibitem{BKOW}
{\sc Michael Braun, Axel Kohnert, Patric R. \"Ostergard and Alfred Wassermann},
{\em Large sets of $t$-designs over finite fields},
J. Combin. Theory Ser. A {\bf 124} (2014), 195--202.

\bibitem{Cohn}
{\sc Henry Cohn},
{\em Projective geometry over $\mathbb F_1$ and the {G}aussian binomial coefficients},
Amer. Math. Monthly {\bf 111} (2004), 487--495.

\bibitem{SBraun10}
{\sc Stefanie Braun},
{\em Construction of $q$-analogs of combinatorial designs},
Alcoma'10, Thurnau, 2010.

\bibitem{Cameron74}
{\sc Cameron, P. J.},
{''Generalization of Fisher's inequality to fields with more than one element''}.
In: {\em Combinatorics. Proceedings of the British Combinatorial Conference 1973},
London Mathematical Society Lecture Note Series 13 (1974), 9--13, Cambridge.

\bibitem{D}
{\sc Leon M. H. E. Driessen},
{\em $t$-designs, $t \ge 3$},
Technical Report, Department of 
Mathematics, Eindhoven University of Technology, 1978.

\bibitem{FLV}
{\sc Arman Fazeli, Shachar Lovett and Alexander Vardi},
{\em Nontrivial $t$-designs over finite fields exist for all $t$},
arXiv: 1306.2088.

\bibitem{GoldmanRota70}
{\sc Jay Goldman and Gian-Carlo Rota},
{\em On the foundations of combinatorial theory. IV. Finite vector spaces and Eulerian generating functions},
Stud. Appl. Math. {\bf 49} (1970), 239--258.

\bibitem{Kiermaier-Pavcevic}
{\sc Michael Kiermaier and Mario Osvin Pavčević},
{\em Intersection numbers for subspace designs},
submitted (2014).

\bibitem{vL}
{\sc Dirk C. van Leijenhorst},
{\em Orbits on the projective line},
J. Comb. Theory, Ser. A {\bf 31} (1981), 146--154.

\bibitem{Schwartz-Etzion-2002}
{\sc Moshe Schwartz and Tuvi Etzion},
{\em Codes and Anticodes in the Grassman Graph},
J. Combin. Theory Ser. A {\bf 97} (2002), 27--42.

\bibitem{suzuki-fivedays}
{\sc Hiroshi Suzuki},
{\em Five Days Introduction to the Theory of Designs},
1989, available online at \url{http://subsite.icu.ac.jp/people/hsuzuki/lecturenote/designtheory.pdf}

\bibitem{Suzuki-1990}
{\sc Hiroshi Suzuki},
{\em On the inequalities of $t$-designs over a finite field},
European J. Combin. {\bf 11} (1990), 601--607.

\bibitem{T}
{\sc Luc Teirlinck},
{\em Non-trivial $t$-designs without repeated blocks exist for all $t$},
Discrete Math. {\bf 65} (1987), 301--311.

\bibitem{Thomas87}
{\sc Simon Thomas},
{\em Designs over finite fields},
Geom. Dedicata {\bf 24} (1987), 237--242.

\bibitem{Tits}
{\sc Jacques Tits},
{\em Sur les analogues alg\'{e}briques des groupes semi-simples complexes,}
in {\em Colloque d'Alg\'ebre Sup\'erieure}, tenue \`a Bruxelles du 19 au 22 d\'ecembre 1956,
Centre Belge de Recherches Math\`ematiques \`Etablissements
Ceuterick, Louvain, Paris: Librairie Gauthiers-Villars, (1957), 261--289.

\bibitem{tvt}
{\sc Tran Van Trung},
{\em On the construction of $t$-designs and the existence of some
             new infinite families of simple $5$-designs},
Arch. Math. (Basel) {\bf 47} (1986), 187--192.
\end{thebibliography}
\end{document}